\documentclass[12pt,reqno]{amsart}
\usepackage{fullpage}
\usepackage{times}
\usepackage{amsmath}
\usepackage{amssymb}
\usepackage{xcolor}
\usepackage{textcomp}
\usepackage{esint}
\usepackage{graphicx}
\usepackage{adjustbox}

\newtheorem{theorem}{Theorem}[section]
\newtheorem{lemma}[theorem]{Lemma}
\newtheorem{proposition}[theorem]{Proposition}

\theoremstyle{definition}
\newtheorem{definition}[theorem]{Definition}

\newtheorem{remark}[theorem]{Remark}
\numberwithin{equation}{section}

\renewcommand{\mod}[1]{{\ifmmode\text{\rm\ (mod~$#1$)}\else\discretionary{}{}{\hbox{ }}\rm(mod~$#1$)\fi}}

\newcommand{\A}{{\mathcal A}}
\newcommand{\B}{{\mathcal B}}
\newcommand{\Blq}{\B_{\ell,q}}
\newcommand{\C}{{\mathbb C}}
\newcommand{\E}{{\mathcal E}}
\newcommand{\cH}{{\mathcal H}}
\newcommand{\Hlq}{\cH_{\ell,q}}
\newcommand{\I}{{\mathcal I}}
\newcommand{\Ilq}{\I_{\ell,q}}
\newcommand{\M}{{\mathcal M}}
\newcommand{\Mlq}{\M_{\ell,q}}
\newcommand{\cN}{{\mathcal N}}
\newcommand{\N}{{\mathbb N}}
\newcommand{\Z}{{\mathbb Z}}
\newcommand{\ord}{\mathop{\rm ord}}
\newcommand{\mpq}{m(p,q)}
\newcommand{\mlq}{m(\ell,q)}

\usepackage[colorlinks=true, pdfstartview=FitH, linkcolor=blue, citecolor=blue, urlcolor=blue]{hyperref}
\DeclareMathAlphabet{\mathpzc}{OT1}{pzc}{m}{it}

\begin{document}

\title{The least primary factor of the multiplicative group}
\author{Greg Martin and Chau Nguyen}
\date{\today}
\address{Department of Mathematics \\ University of British Columbia \\ Room 121, 1984 Mathematics Road \\ Vancouver, BC, Canada \ V6T 1Z2}
\email{gerg@math.ubc.ca}
\email{nguyenminhchauptnk@gmail.com}
\subjclass[2010]{11N25, 11N37, 11N45, 11N64, 20K01}
\maketitle

\begin{abstract}
Let $S(n)$ denote the least primary factor in the primary decomposition of the multiplicative group $M_n = (\Bbb Z/n\Bbb Z)^\times$. We give an asymptotic formula, with order of magnitude $x/(\log x)^{1/2}$, for the counting function of those integers $n$ for which $S(n) \ne 2$. We also give an asymptotic formula, for any prime power~$q$, for the counting function of those integers $n$ for which $S(n) = q$. This group-theoretic problem can be reduced to problems of counting integers with restrictions on their prime factors, allowing it to be addressed by classical techniques of analytic number theory.
\end{abstract}

\section{Introduction}

The {\em multiplicative group} modulo~$n$ is the finite abelian group $M_n = (\Z/n\Z)^\times$ of units in the ring $\Z/n\Z$, or equivalently the group of reduced residue classes modulo~$n$ under multiplication. Several common arithmetic functions are related to the multiplicative group: its size is the Euler phi-function $\phi(n)$, for instance, and its exponent is the Carmichael lambda-function $\lambda(n)$. Both of these functions have been investigated for their distributional properties (as well as pointwise bounds), which reflects the fact that the multiplicative groups~$M_n$ form a natural family of groups to study, and indeed a richer family of groups than the cyclic (additive) groups~$\Z_n$ themselves.

The structure of these multiplicative groups is best described in terms of one of the two canonical forms for finite abelian groups~$G$. One canonical form is the invariant factor decomposition $G \cong \Z_{d_1} \times \cdots \times \Z_{d_k}$ where each $d_j$ divides $d_{j+1}$; indeed in this form we have $d_k = \lambda(n)$ exactly. It turns out that the length~$k$ of the invariant factor decomposition is essentially $\omega(n)$, the number of distinct prime factors of~$n$, whose distribution also has a robust theory. Chang and the first author~\cite{CM} used a refinement of the Selberg--Delange method to show that $d_1=2$ for almost all integers~$n$ and found an asymptotic formula for the number of $n\le x$ for which $d_1=d$ for any fixed integer~$d$.

In this paper we focus on the other canonical form for finite abelian groups, namely the {\em primary decomposition} $G \cong \Z_{q_1} \times \cdots \times \Z_{q_t}$ where each~$q_j$ is a prime power; this decomposition (which is also called the elementary factor decomposition) is uniquely determined by~$G$ up to the order of the primary factors~$q_j$. The result mentioned above implies that the least primary factor of $M_n$ equals~$2$ for almost all integers~$n$ (since every primary factor of an invariant factor of~$G$ is also a primary factor of~$G$ itself). More specifically, we can prove:

\begin{theorem} \label{exceptions theorem}
The number of $n\le x$ for which~$2$ is not the least primary factor of $M_n$ is
\[
C_3 \frac{x}{(\log x)^{1/2}} + O \biggl( \frac x{(\log x)^{2/3}} \biggr),
\]
where
\begin{equation} \label{C3}
C_3 = \frac3{4\sqrt2} \prod_{p \equiv 3 \mod 4}\bigg(1 - \frac{1}{p^2}\bigg)^{{1}/{2}}  \approx 0.490694.
\end{equation}
\end{theorem}

The main result of this paper is to use the same refinement of the Selberg--Delange method to establish an asymptotic formula for the number of integers $n\le x$ for which the least primary factor of $M_n$ equals any fixed prime power~$q$. We introduce the following notation:

\begin{definition} \label{definition 1.1}
Let $S(n)$ denote the smallest primary factor of~$M_n$, that is, the least prime power appearing in the primary decomposition of~$M_n$. (Note that $M_1$ and $M_2$ are trivial groups; we are content to leave $S(1)$ and $S(2)$ undefined.) Also let $\E_q(x)$ denote the set of positive integers $n \leq x$ such that $S(n) = q$. 
\end{definition}

\begin{theorem} \label{main theorem}
For all prime powers~$q\le(\log x)^{1/3}$,
\[
\#\E_q(x) = \frac{C_q x}{(\log x)^{1 - \beta_q}} + O_q \biggl( \frac x{(\log x)^{1 - \beta_{q^+}}} \biggr),
\]
where the constants~$\beta_q$, $\beta_{q^+}$, and~$C_q$ are defined in Definitions~\ref{beta_q definition} and~\ref{leading constant definition}.
\end{theorem}

The constant $C_q$ has an implicit definition but can be computed in principle for any~$q$; we work out the following special cases as examples.

\begin{theorem} \label{q=3,4,5 theorem}
We have
\begin{align*}
\#\E_3(x) &= C_3 \frac{x}{(\log x)^{1/2}} + O \biggl( \frac x{(\log x)^{2/3}} \biggr) \\
\#\E_4(x) &= C_4 \frac{x}{(\log x)^{2/3}} + O \biggl( \frac x{(\log x)^{5/6}} \biggr) \\
\#\E_5(x) &= C_5 \frac{x}{(\log x)^{5/6}} + O \biggl( \frac x{(\log x)^{13/15}} \biggr)
\end{align*}
where $C_3$ was defined in equation~\eqref{C3} above, and $C_4 \approx 0.4200344$ and $C_5 \approx 0.2095134$ are given by
\begin{multline} \label{C4}
C_4 = \frac{3G_{\B_4}(1)}{2\Gamma(\beta_4)} = \frac{3^{2/3}}{2\Gamma(1/3)} \prod_{\substack{\chi \ne \chi_0 \\ \chi_3=\chi_0}} L(1,\chi)^{1/3} \prod_{\substack{\chi \ne \chi_0 \\ \chi_3^2=\chi_0}} L(1,\chi)^{-1/6} \prod_{\chi \ne \chi_0} L(1,\chi)^{1/12} \\
\times \prod_{p \equiv 5,29 \mod{36}} (1-p^{-6})^{1/6} (1-p^{-2})^{-1/2} \prod_{p \equiv 7,11,23,31 \mod{36}} (1-p^{-6})^{1/6} \\
\times \prod_{p \equiv 13,25 \mod{36}} (1-p^{-3})^{1/3} \prod_{p \equiv 19,35 \mod{36}} (1-p^{-2})^{1/2}
\end{multline}
and
\begin{multline} \label{C5}
C_5 = \frac{3^{5/6}}{2\Gamma(1/6)} \prod_{\substack{\chi \ne \chi_0 \\ \chi_3=\chi_0}} L(1,\chi)^{1/6} \prod_{\substack{\chi \ne \chi_0 \\ \chi_3^2=\chi_0}} L(1,\chi)^{-1/12} \prod_{\chi \ne \chi_0} L(1,\chi)^{1/24} \\
\prod_{\substack{p \equiv 19,35,37,53, \\ 55,71 \mod {72}}} (1-p^{-2})^{1/2} \prod_{\substack{p \equiv 5,7,11,13,23,29,31,43, \\ 47,59,61,67 \mod {72}}} (1-p^{-6})^{1/6} \\
\times \prod_{p \equiv 41,65 \mod {72}} (1-p^{-6})^{1/6} (1-p^{-2})^{1/2} \prod_{p \equiv 25,49 \mod {72}} (1-p^{-3})^{1/3}.
\end{multline}
\end{theorem}

\begin{remark}
The notation $\chi_3$ appearing in the expressions for~$C_4$ and~$C_5$ will be explained more fully in Remark~\ref{Bq Blq remark} below. For now, it suffices to say that~$\chi$ in the expression for~$C_4$ runs over all nonprincipal Dirichlet characters modulo~$36$, and any such character restricted to the residue class $1\mod4$ yields a Dirichlet character modulo~$9$; it is this latter character that~$\chi_3$ refers to. In the expression for~$C_5$, the same convention holds with~$36$ replaced by~$72$ and with $1\mod4$ replaced by $1\mod8$.
\end{remark}

Throughout this paper, $p$ and $\ell$ will always refer to primes, while $q$ will refer to a prime power (possibly itself prime). All $O$-constants may depend on~$q$. In Section~\ref{init section} we prove Theorem~\ref{exceptions theorem}; the argument also serves as a foundation for the subsequent, more involved proof of Theorem~\ref{main theorem}. In Section~\ref{convert section} we characterize integers~$n$ for which $M_n$ has no elementary factors less than~$q$ in terms of the factorization of~$n$, and then deduce Theorem~\ref{main theorem} from that characterization using the Selberg--Delange method. The rest of the paper is devoted to demonstrating how to calculate the leading constants in Theorem~\ref{main theorem} explicitly in a form that allows for numerical approximation; the overall plan for doing so is given in Section~\ref{plan section}, and the evaluation of~$C_3$ is carried out in Section~\ref{C3 section}. After evaluating some subsidiary sums and products involving Dirichlet characters in Section~\ref{char section}, we complete the proof of Theorem~\ref{q=3,4,5 theorem} by evaluating~$C_4$ in Section~\ref{C4 section} and~$C_5$ in Section~\ref{C5 section}.

\section{Counting integers $n$ with least primary factor not equal to~$2$} \label{init section}

In this section we prove Theorem~\ref{exceptions theorem}; the argument also serves as a foundation for the more involved proof of Theorem~\ref{main theorem} in the next section.

We rely in an essential way on the fact that the primary decomposition of a direct product $G = \prod_{j=1}^k G_k$ of finite abelian groups is simply the concatenation of the primary decompositions of the individual groups~$G_j$. In particular, the least primary factor of~$G$ is equal to the minimum of the least primary factors of the~$G_j$. This observation gives a particularly simple criterion for detecting the smallest possible primary factor~$2$. Recall that $S(n)$ denotes the smallest primary factor of the multiplicative group~$M_n$. We also write $p^r\|n$ to mean $p^r\mid n$ but $p^{r+1}\nmid n$.

\begin{lemma} \label{lemma 3.1}
We have $S(n) \ne 2$ if and only if both $4 \nmid n$ and all odd primes dividing~$n$ are congruent to $1\mod 4$.
\end{lemma}

\begin{proof}
If the prime-power factorization of~$n$ is $n = \prod_{p^r \mid \mid n} p^r$, the Chinese remainder theorem tells us that $M_n \cong \prod_{p^r \| n} M_{p^r}$; in particular, $S(n)=2$ if and only if $S(p^r)=2$ for some prime power $p^r\| n$.

When~$p$ is odd, classical results on primitive roots imply that $M_{p^r} \cong \Z_{\phi(p^r)} \cong \Z_{p^{r-1}} \oplus \Z_{p-1}$; the least primary factor of this direct sum is the least primary factor of $\Z_{p-1}$, which is the smallest prime power dividing $p-1$. In particular, its least primary factor equals~$2$ precisely when $p\equiv3\mod 4$.

On the other hand, $M_2$ is the trivial group, while $M_{2^r} \cong \Z_2 \oplus \Z_{2^{r-2}}$ when $r\ge2$. In particular, the least primary factor of $M_{2^r}$ equals~$2$ as soon as $r\ge2$.
\end{proof}

As we carry out the proof of Theorem~\ref{exceptions theorem}, we introduce notation that will foreshadow more general notation in the next section.

\begin{definition} \label{definition 3.2}
Let $\A_3(x)$ be the set of all positive integers~$n$ not exceeding~$x$ such that $S(n) \ge 3$. Let $\A'_3(x)$ be the set of all odd integers in $\A_3(x)$.
\end{definition}

Note that $\{n\le x\colon n \notin \E_2(x)\} = \{n\le x\colon S(n) \ne 2\} = \{n\le x\colon S(n)\ge 3\} = \A_3(x)$, and so Theorem~\ref{exceptions theorem} is a claim about an asymptotic formula for $\#\A_3(x)$, which we will deduce from an asymptotic formula for $\#\A'_3(x)$.

\begin{lemma} \label{split A3}
We have $\#\A_3(x) = \#\A'_3(x) + \#\A'_3(\frac{x}{2})$.
\end{lemma}

\begin{proof}
The odd elements of $\A_3(x)$ are counted by $\#\A'_3(x)$ by definition. On the other hand, every even $n\in\A_3(x)$ is of the form $n=2m$ where~$m \le \frac x2$ is odd, since $4\nmid n$ by Lemma~\ref{lemma 3.1}; therefore the even elements of $\A_3(x)$ are counted by $\#\A'_3(\frac x2)$.
\end{proof}

Note that $\A_3'(x) = \{n \leq x\colon p\mid n \implies p\equiv1\mod 4\}$ by Lemma~\ref{lemma 3.1}. Counting integers with this sort of restriction on their prime factors is a classical problem in analytic number theory (and indeed this particular case is almost the same as the seminal study by Landau ~\cite{Landau} of the counting function of sums of two squares).

\begin{lemma} \label{A3 asymptotic}
We have $\displaystyle\#\A'_3(x) = \frac{g_4x}{(\log x)^{1/2}} + O\bigg(\frac{x}{(\log x)^{3/4}}\bigg)$, where
\[
g_4 = \frac1{2\sqrt2} \prod_{p \equiv 3 \mod 4}\bigg(1 - \frac{1}{p^2}\bigg)^{{1}/{2}} \approx 0.32713.
\]
\end{lemma}

\begin{proof}
The paper of Chang and first author~\cite[Proposition~4.1]{CM} contains the asymptotic formula
\[
\#\{n \leq x\colon p\mid n \implies p\equiv1\mod q\} = \frac{g_qx}{(\log x)^{1 - {1}/{\phi(q)}}} + O\bigg(\frac{x}{(\log x)^{{5}/{4} - {1}/{\phi(q)}}}\bigg)
\]
for any even positive integer $q \geq 3$, where
\begin{equation} \label{G example}
g_q = \frac{1}{\Gamma({1}/{\phi(q)})}\bigg(\frac{\phi(q)}{q}{\prod_{\substack{\chi\mod q \\ \chi \ne \chi_0}}L(1,\chi)}\bigg)^{{1}/{\phi(q)}}{\prod_{\substack{p \nmid q \\ p \not \equiv 1 \mod q}}} \bigg( 1 - \frac{1}{p^{\ord_q(p)}} \bigg)^{{1}/{\ord_q(p)}}.
\end{equation}
Here $\ord_q(a)$ is the order of~$a$ in the multiplicative group $M_q$, and $\Gamma(s)$ and $L(s,\chi)$ are the Euler Gamma function and Dirichlet $L$-functions, respectively. The simplification of the constant~$g_4$ was carried out after Theorem~1.2 of that paper:
\begin{align}
g_4 &= \frac{1}{\Gamma({1}/{2})}\bigg(\frac{2}{4}L(1, \chi_{-4})\bigg)^{{1}/{2}}{\prod_{p \equiv 3 \mod 4}\bigg(1 - \frac{1}{p^2}\bigg)^{{1}/{2}}} \notag \\
&=\frac{1}{\sqrt{\pi}}\Big(\frac{\pi}{8}\Big)^{{1}/{2}} \prod_{p \equiv 3 \mod 4}\bigg(1 - \frac{1}{p^2}\bigg)^{{1}/{2}}  \label{convergent Euler}
\end{align}
as claimed.
\end{proof}

With these results in hand, the proof of Theorem~\ref{exceptions theorem} is a simple calculation:

\begin{proof}[Proof of Theorem~\ref{exceptions theorem}]
By Lemmas~\ref{split A3} and~\ref{A3 asymptotic},
\begin{align*}
\#\A_3(x) &= \#\A'_3(x) + \#\A'_3(\tfrac{x}{2}) \\
&= \bigg(\frac{g_4x}{(\log{x})^{1 / 2}} + O\bigg(\frac{x}{(\log x)^{3/4}}\bigg)\bigg) + \bigg(\frac{g_4x/2}{(\log{x/2})^{1 / 2}} + O\bigg(\frac{x/2}{(\log x/2)^{3/4}}\bigg)\bigg)\\
&= g_4x\bigg(\frac{1}{(\log x)^{1/2}} + \frac{1}{2(\log ({x}/{2}))^{{1}/{2}}}\bigg) + O\bigg(\frac{x}{(\log x)^{{3}/{4}}}\bigg) \\
&= g_4x\bigg(\frac{1}{(\log x)^{1/2}} + \frac{1}{2(\log x)^{{1}/{2}}} \bigg\{ 1 + O\bigg( \frac1{\log x} \bigg) \bigg\} \bigg) + O\bigg(\frac{x}{(\log x)^{{3}/{4}}}\bigg) \\
&= g_4x\bigg(\frac{3/2}{(\log x)^{1/2}}\bigg) + O\bigg(\frac{x}{(\log x)^{{3}/{4}}}\bigg),
\end{align*}
which is the desired asymptotic formula since $C_3 = \frac32g_4$.
\end{proof}

\section{Counting integers $n$ with least primary factor equal to~$q$} \label{convert section}

Recall that $\E_q(x)$ from Definition~\ref{definition 1.1} is comprised of integers~$m$ such that the smallest primary factor $S(m)$ is exactly equal to~$q$, while $\A_q(x)$ from Definition~\ref{definition 3.2} is comprised of integers~$m$ such that $S(m)$ is at least~$q$.
In this section we prove Theorem~\ref{main theorem}, the asymptotic formula for $\#\E_q(x)$, using an elaboration on the method for finding an asymptotic formula for $\#\A_3(x)$ in the previous section.

\begin{definition} \label{q plus definition}
Given a prime power $q$, let $q^+$ denote the smallest prime power that is greater than~$q$.
Note that
\begin{align} \notag
\#\E_q(x) &= \{ m \leq x \colon S(m) = q \} \\
&= \{ m \leq x \colon S(m) \ge q \} - \{ m \leq x \colon S(m) \ge q^+ \} = \#\A_q(x) - \#\A_{q^+}(x). \label{E=A-A}
\end{align}
Thus to obtain an asymptotic formula for $\#\E_q(x)$ for all prime powers~$q$, it suffices to obtain an asymptotic formula for $\#\A_q(x)$ for all prime powers~$q$. This observation is helpful because, as it turns out, the property $S(m) \ge q$ is easier to characterize than the property $S(m) = q$ (see Proposition~\ref{prop 4.7} below).
\end{definition}

\begin{definition} \label{definition 4.3}
Given a prime power $q$ and a prime $p<q$, let $\mpq$ denote the largest positive integer such that $p^{\mpq} < q$, so that $\mpq = \lceil (\log q) / \log p \rceil - 1$.
\end{definition}

\begin{definition} \label{definition 4.4}
For any prime~$p$, let $v_p(n)$ denote the multiplicity with which~$p$ divides~$n$. For example, $v_2(40) = 3$ and $v_3(21)=1$ and $v_3(5) = 0$.
\end{definition}

\begin{lemma} \label{buildup lemma 1}
Fix primes $\ell$ and $p$ and a prime power~$q\ge3$. There exists $s\in\N$ such that $\ell^s<q$ and $\ell^s\|(p-1)$ if and only if $p\equiv1\mod\ell$ and $p\not\equiv1\mod{\ell^{m(\ell,q)+1}}$.
\end{lemma}

\begin{proof}
Choose $s=v_\ell(p-1)$, the only possibility for satisfying $\ell^s\|(p-1)$. If $s=0$ then the equivalence holds because both $s\notin\N$ and $p\not\equiv1\mod\ell$. If $s>m(\ell,q)$ then the equivalence holds as well because both $\ell^s\ge q$ and $p\equiv1\mod{\ell^{m(\ell,q)+1}}$. Otherwise $1\le s\le m(\ell,q)$, and the equivalence holds because all statements are true in this case.
\end{proof}

\begin{lemma} \label{buildup lemma 2}
Fix an odd prime power~$p^k$ and a prime power $q\ge3$. We have $S(p^k)\ge q$ if and only if for every prime $\ell<q$, either $p\not\equiv1\mod\ell$ or $p\equiv1\mod{\ell^{m(\ell,q)+1}}$.
\end{lemma}

\begin{proof}
By Lemma~\ref{buildup lemma 1}, it suffices to show that $S(p^k)<q$ if and only if there exists a prime~$\ell$ and an integer $s\in\N$ such that $\ell^s<q$ and $\ell^s\|(p-1)$. But this is clear because $S(p^k)$ is the smallest prime power $\ell^j$ such that $\ell^j\|(p-1)$, since $M_{p^k}$ is cyclic of order $p^{k-1}(p-1)$. (We remark that the exponent~$k$ plays no role, as the second half of the claimed equivalence indicates.)
\end{proof}

\begin{proposition} \label{prop 4.7}
Let $q\ge3$ be a prime power and $m\ge3$ an integer. We have $S(m) \geq q$ if and only if:
\begin{enumerate}
    \item $4 \nmid m$;
    \item if $p\le q$ is an odd prime, then $p\nmid m$; and
    \item if $\ell<q$ is a prime and $p\mid m$ is an odd prime, then either $p \not \equiv 1 \mod \ell$ or $p \equiv 1\mod{\ell^{m(\ell,q) + 1}}$.
\end{enumerate}
\end{proposition}

\noindent
(Technically part~(b) is a special case of part~(c), since there is always some prime divisor~$\ell$ of $p-1$ and this~$\ell$ will automatically be less than~$q$ if $p\le q$. But it will be helpful to record part~(b) explicitly.)

\begin{proof}
Since $M_m \cong \prod_{p^k\|m} M_{p^k}$ by the Chinese remainder theorem, we see that $S(m)$ is the minimum value of $S(p^k)$ over all $p^k\|m$; in particular, $S(m) \ge q$ if and only if $S(p^k) \ge q$ for all $p^k\|m$. We check the implications of this inequality separately for $p=2$, for $3\le p\le q$, and for $p>q$.
\begin{enumerate}
\item If $2^k\|m$ with $k\ge2$, then $M(2^k) \cong \Z_2\times\Z_{2^{k-2}}$ and thus $S(2^k)=2<q$. Otherwise $p=2$ does not contribute to $M(m)$.
\item If $p\mid m$ for some $3\le p\le q$, then $S(p) \le p-1 < q$, which makes $S(m) \ge q$ impossible.
\item This last item follows directly from Lemma~\ref{buildup lemma 2}.
\qedhere
\end{enumerate}
\end{proof}

\begin{remark} \label{counting residue classes}
Recall that $\mlq$ was defined in Definition~\ref{definition 4.3}. Proposition~\ref{prop 4.7}(c) says that for each prime $\ell<q$, each odd prime $p\mid m$ belongs to one of $(\ell - 2)\ell^{\mlq} + 1$ reduced residue classes modulo $\ell^{\mlq + 1}$: one of those residue classes is $1\mod{\ell^{\mlq} + 1}$, and the other $(\ell - 2)p^{\mlq}$ residue classes modulo $\ell^{\mlq + 1}$ are those congruent to neither~$0$ nor~$1\mod\ell$. (We write these residue classes explicitly in Definition~\ref{Bpq def} below.)
\end{remark}



\begin{definition} \label{beta_q definition}
Given a prime power $q\ge3$, define
\[
Q_q = \prod_{\ell < q} \ell^{\mlq + 1}
\quad\text{and}\quad
B_q = \prod_{\ell < q} \bigl( (\ell-2)\ell^{\mlq} + 1 \bigr),
\]
and define $\B_q$ to be the union of~$B_q$ reduced residue classes modulo~$Q_q$, namely those corresponding to the $(\ell-2)\ell^{\mlq} + 1$ residue classes modulo $\ell^{\mlq + 1}$ described in Remark~\ref{counting residue classes} for each $\ell<q$. Further define
\begin{equation} \label{beta_q definition equation}
\beta_q = \frac{B_q}{\phi(Q_q)} = {\prod_{\ell < q} \biggl( \frac{\ell - 2}{\ell - 1} + \frac{1}{\ell^{\mlq}(\ell-1)}} \biggr).
\end{equation}
Recalling from Definition~\ref{q plus definition} that $q^+$ is the next prime power after~$q$, we see thar
\[
\beta_{q^+} = {\prod_{\ell \le q} \biggl( \frac{\ell - 2}{\ell - 1} + \frac{1}{\ell^{\mlq}(\ell-1)}} \biggr).
\]
\end{definition}

\begin{remark} \label{beta q plus}
If~$q$ is prime then $\beta_{q^+} = \frac{q-1}q \beta_q$; if $q=\ell^m$ with $m\ge2$ then
\[
\beta_{q^+} = \frac{\ell-2+\ell^{-m}}{\ell-2+\ell^{1-m}} \beta_q.
\]
Either way, we see that $\beta_{q^+} < \beta_q$.
\end{remark}

\begin{definition} \label{notation 2.1}
Let $Q \geq 3$ be an integer. Let $\B$ be the union of $B$ distinct reduced residue classes\mod Q, and set $\beta = {B}/{\phi(Q)}$. Let $\cN_\B = \{n \in \N \colon p \mid n \Rightarrow p \in \B\}$ be the set of all positive integers~$n$ whose prime factors all lie in~$\B$. We define the associated Dirichlet series:
\[
F_\B(s) = \sum_{n \in \cN_\B} \frac{1}{n^s} = \prod_{p \in \B} \frac{1}{1 - p^{-s}},
\]
which converges absolutely when $\Re(s) > 1$. If we define
\[
G_\B(s) = F_\B(s)  \zeta(s)^{-\beta},
\]
it turns out that $G_\B(s)$ has a removable singularity at $s=1$, and thus $G_\B(1)$ is well defined.
\end{definition}

Recall that  $\A'_q(x)$ from Definition~\ref{definition 3.2} is comprised of odd integers~$m$ such that $S(m)$ is at least~$q$.

\begin{proposition} \label{lemma 4.10}
For any prime power $q\ge3$,
\[
\#\A'_q(x) = \frac{x}{{(\log x)}^{1 - \beta_q}} \biggl( \frac{G_{\B_q}(1)}{\Gamma(\beta_q)} + O((\log x)^{-{1}/{4}}) \biggr).
\]
\end{proposition}

\begin{proof}
By Proposition~\ref{prop 4.7}, we have $\A'_q(x) = \{n \leq x \colon n \in \cN_{\B_q} \}$, where $\B_q$ was defined in Definition~\ref{beta_q definition}. The proposition then follows directly from the asymptotic formula
\[
\#\{n \leq x \colon n \in \cN_\B \} = \frac{x}{(\log x)^{1 - \beta}}\biggl( \frac{G_\B(1)}{\Gamma(\beta)} + O\bigl( (\log x)^{-{1}/{4}} \bigr) \biggr)
\]
derived by Chang and the second author~\cite[Theorem~3.4]{CM} (in particular, that derivation verifies the claim that $G_\B(s)$ has a removable singularity at $s=1$).
\end{proof}

Lemma~\ref{split A3} is straightforward to generalize:

\begin{lemma} \label{split Aq}
We have $\#\A_q(x) = \#\A'_q(x) + \#\A'_q(\frac{x}{2})$.
\end{lemma}

\begin{proof}
The odd elements of $\A_q(x)$ are counted by $\#\A'_q(x)$ by definition. On the other hand, every even $n\in\A_q(x)$ is of the form $n=2m$ where~$m \le \frac x2$ is odd, since $4\nmid n$ by Proposition~\ref{prop 4.7}; therefore the even elements of $\A_q(x)$ are counted by $\#\A'_q(\frac x2)$.
\end{proof}

\begin{definition} \label{leading constant definition}
Using the notation $\B_q$ and $\beta_q$ from Definition~\ref{beta_q definition} and the function~$G$ from Definition~\ref{notation 2.1}, define
\[
C_q = \frac{3G_{\B_q}(1)}{2\Gamma(\beta_q)} = \frac3{2\Gamma(\beta_q)} \lim_{s\to1} \biggl( \zeta(s)^{-\beta_q} \prod_{p\in\B_q} \frac1{1-p^{-s}} \biggr),
\]
where $\Gamma$ is the Euler Gamma-function.
\end{definition}

\begin{proof}[Proof of Theorem~\ref{main theorem}]
By equation~\eqref{E=A-A} and Lemma~\ref{split Aq},
\begin{align*}
\#\E_q(x) &= \#\A_q(x) - \#\A_{q^+}(x) \\ 
&= \#\A'_q(x) + \#\A'_q(\tfrac{x}{2}) + O\bigl( \#\A'_{q^+}(x) + \#\A_{q^+}(\tfrac{x}{2}) \bigr) \\ 
&= \frac{x}{(\log x)^{1 - \beta_q}}\bigg(\frac{G_{\B_q}(1)}{\Gamma(\beta_q)} + O({(\log x)}^{-{1}/{4}})\bigg) \\
&\qquad{}+ \frac{{x}/{2}}{\log ({x}/{2})^{1 - \beta_q}}\bigg(\frac{G_{\B_q}(1)}{\Gamma(\beta_q)} + O((\log {x}/{2})^{-{1}/{4}})\bigg) + O\biggl( \frac{x}{(\log x)^{1 - \beta_{q^+}}} \biggr) \\
&= \frac32 \frac{G_{\B_q}(1)}{\Gamma(\beta_q)} \frac{x}{(\log x)^{1 - \beta_q}} + O\biggl( \frac{x}{(\log x)^{1 - \beta_{q^+}}} \biggr),
\end{align*}
which completes the proof of Theorem~\ref{main theorem} by virtue of Definition~\ref{leading constant definition}.
\end{proof}

\section{Plan for evaluating the leading constants} \label{plan section}

The remainder of this paper analyzes the constant~$C_q$ from Definition~\ref{leading constant definition} for $q\in\{3,4,5\}$, establishing Theorem~\ref{q=3,4,5 theorem} and illustrating that in principle $C_q$ can be given a more concrete form for any prime power~$q$. The difficult component of~$C_q$ is the constant $G_{\B_q}(1)$, which we now outline a method for evaluating.

Recall Definition~\ref{beta_q definition} for the quantities~$Q_q$ and~$B_q$ and the set~$\B_q$, as well as Definition~\ref{notation 2.1} for the functions~$F_\B$ and~$G_\B$. Chang and the first author showed~\cite[equation~(3.3)]{CM} that
\begin{equation} \label{GBqs formula}
G_{\B_q}(s)^{\phi(Q_q)} = A_{\B_q}(s) {\prod_{p \mid Q_q}(1 - p^{-s})^{B_q} \prod_{\substack{\chi \mod{Q_q} \\ \chi \ne \chi_0}} L(s,\chi)^{{\sum_{b \in \B_q}\overline{\chi}(b)}}}.
\end{equation}
In the two products we can simply plug in $s=1$ and evaluate, using equations~\eqref{primitive L 1 chi} and~\eqref{imprimitive L 1 chi} below for the $L(1,\chi)$ terms. (It turns out that the exponent of $L(s,\chi)$ is an integer---see Lemma~\ref{lemma 5.1} below---and in particular real-valued. Therefore the right-hand side is a positive real number at $s=1$, as we can see after pairing the terms~$\chi$ and~$\overline\chi$ in the product; thus taking the $\phi(Q_q)$th root to solve for $G_{\B_q}(1)$ is unproblematic.) As for the $A_{\B_q}(s)$ factor, Chang and the first author further showed~\cite[equation~(3.1)]{CM}~that
\begin{equation} \label{ABqs formula}
A_{\B_q}(s) = F_{\B_q}(s)^{\phi(Q_q)} \prod_{\chi \mod{Q_q}} L(s,\chi)^{-\sum_{b \in \B_q}\overline{\chi}(b)}
\end{equation}
To evaluate $A_{\B_q}(1)$, we will rewrite the right-hand side in terms of convergent Euler products over primes in various residue classes, such as the one arising in equation~\eqref{Euler example} whose value at $s=1$ already appeared in equation~\eqref{convergent Euler}.

\section{Calculating the leading constant $C_3$} \label{C3 section}

Implementing the plan from the previous section when $q=3$ is essentially reproducing the calculation summarized in the proof of Lemma~\ref{A3 asymptotic}. Carrying out this computation explicitly will be helpful practice before addressing the more complicated cases $q=4$ and $q=5$ in later sections.

\begin{proof}[Proof of Theorem~\ref{q=3,4,5 theorem} for $q=3$]
We have $Q_3 = \prod_{\ell<3} \ell^{\mlq+1} = 2^{1+1} = 4$ and $\B_3 = \{1\mod4\}$. Equation~\eqref{ABqs formula} becomes
\begin{align*}
A_{\B_3}(s) &= F_{\B_3}(s)^{\phi(Q_3)} \prod_{\chi\mod4} L(s,\chi)^{-\overline{\chi}(1)} \\
&= \prod_{p\equiv1\mod4} (1-p^{-s})^{-2} \cdot L(s,\chi_0)^{-1} L(s,\chi_{-4})^{-1} \\
&= \prod_{p\equiv1\mod4} (1-p^{-s})^{-2} \cdot \prod_{p\ne2} (1-p^{-s}) (1-\chi_{-4}(p)p^{-s}),
\end{align*}
where $\chi_0$ and $\chi_{-4}$ are the principal and nonprincipal characters modulo~$4$, respectively. Splitting the odd primes into their residue classes modulo~$4$, this formula becomes
\begin{align}
A_{\B_3}(s) &= \prod_{p\equiv1\mod4} (1-p^{-s})^{-2} (1-p^{-s}) (1-p^{-s}) \cdot \prod_{p\equiv3\mod4} (1-p^{-s}) (1+p^{-s}) \notag \\
&= \prod_{p\equiv3\mod4} (1-p^{-2s}). \label{Euler example}
\end{align}
Notice that this product converges whenever $\Re s>\frac12$; in particular we can evaluate $A_{\B_3}(1)$ by simply plugging in $s=1$ to the right-hand side.

Now equation~\eqref{GBqs formula} becomes
\begin{align*}
G_{\B_3}(s)^2 &= \prod_{p\equiv3\mod4} (1-p^{-2s}) \cdot (1 - 2^{-s})^1 \prod_{\substack{\chi \mod4 \\ \chi \ne \chi_0}} L(s,\chi)^{\overline\chi(1)}
\end{align*}
and therefore, using equation~\eqref{Euler example},
\begin{align*}
G_{\B_3}(1) &= \biggl( \tfrac12 L(1,\chi_{-4}) \prod_{p\equiv3\mod4} (1-p^{-2}) \biggr)^{1/2} = \sqrt{\frac\pi8} \prod_{p\equiv3\mod4} (1-p^{-2})^{1/2}.
\end{align*}
Finally, since $\beta_3=\frac12$ and $\beta_{3^+}=\beta_4=\frac13$, Theorem~\ref{main theorem} becomes
\[
\#\E_3(x) = \frac{C_3 x}{(\log x)^{1/2}} + O \biggl( \frac x{(\log x)^{2/3}} \biggr),
\]
where by Definition~\ref{leading constant definition},
\begin{equation} \label{ABs0}
C_3 = \frac3{2\Gamma(1/2)} G_{\B_3}(1) = \frac3{2\sqrt\pi} \sqrt{\frac\pi8} \prod_{p\equiv3\mod4} (1-p^{-2})^{1/2}
\end{equation}
which is consistent with equation~\eqref{C3}.
\end{proof}

\section{Evaluating ${\sum_{b \in \B}\chi(b)}$} \label{char section}

In order to evaluate the exponents in equation~\eqref{ABqs formula}, we need to determine ${\sum_{b \in \B_q}\overline{\chi}(b)}$ in general. We accomplish this task in Lemma~\ref{lemma 5.1} below after introducing some necessary notation.

\begin{definition} \label{Bpq def}
Let $q\ge3$ be a prime power and~$\ell<q$ be a prime. With $\mlq$ as in Definition~\ref{definition 4.3}:
\begin{itemize}
\item Define $\Mlq = M_{\ell^{\mlq+1}}$ to be the full multiplicative group modulo $\ell^{\mlq+1}$.
\item Define $\Hlq = \{b \in M_{\ell^{\mlq + 1}} \colon b \equiv 1 \mod \ell\}$, which is also the subgroup of $(\ell-1)$st powers in $\Mlq$.
\item Let $\Ilq = \{ 1\mod{\ell^{\mlq+1}} \}$ be the singleton set containing only one residue class.
\item Set $\Blq = (\Mlq \setminus \Hlq) \cup \Ilq$, which is the set of residue classes described in Remark~\ref{counting residue classes}.
\end{itemize}
\end{definition}

\begin{remark} \label{Bq Blq remark}
Considering the large multiplicative group $M_{Q_q}$ where $Q_q$ is from Definition~\ref{beta_q definition}, we see that $M_{Q_q} \cong \prod_{\ell<q} \Mlq$ by the Chinese remainder theorem; under this isomorphism, the collection $\B_q$ of residue classes in $M_{Q_q}$ from Definition~\ref{beta_q definition} can be written exactly as $\B_q = \prod_{\ell<q} \Blq$.

Moreover, every Dirichlet character $\chi\mod{Q_q}$ can be uniquely written as the product of Dirichlet characters modulo $\ell^{\mlq+1}$ as~$\ell$ ranges over the primes less than~$q$; more precisely, there exist characters $\{\chi_\ell\colon\ell<q\}$ such that $\chi(b) = \prod_{\ell<q} \chi_\ell(b_\ell)$, where the $b_\ell$ are the components of the image of $b\in M_{Q_q}$ under the isomorphism from $M_{Q_q} \cong \prod_{\ell<q} \Mlq$. It follows, for instance,~that
\begin{equation} \label{decomposing chi eq}
\sum_{b\in \B_q} \chi(b) = \prod_{\ell<q} \sum_{b_\ell\in \Blq} \chi_\ell(b_\ell).
\end{equation}
\end{remark}

\begin{lemma} \label{lemma 5.1}
Let $q\ge3$ be a prime power and let $\ell<q$ be a prime. For any Dirichlet character $\chi\mod{\ell^{\mlq+1}}$,
\[
\sum_{b \in \Blq}\chi(b) = \begin{cases}
\ell^{\mlq}(\ell-1), & \text{if } \chi=\chi_0, \\
0, & \text{if } \chi\ne\chi_0
\end{cases} \Biggr\} - \begin{cases}
\ell^{\mlq}, & \text{if } \chi^{\ell-1}=\chi_0, \\
0, & \text{if } \chi^{\ell-1}\ne\chi_0
\end{cases} \Biggr\} + 1.
\]
In particular, if $\ell=2$ then $\sum_{b \in \Blq}\chi(b)=1$ for any $\chi\mod{\ell^{\mlq+1}}$.
\end{lemma}

\begin{proof}
We concentrate on the first assertion since the second one follows directly from it (or directly from the fact that $\Blq=\{1\mod{\ell^{\mlq+1}}\}$ when $\ell=2$).
By definition,
\begin{align*}
{\sum_{b \in \Blq}\chi(b)} &= {\sum_{b \in \Mlq}\chi(b)} - {\sum_{b \in \Hlq}\chi(b)} + {\sum_{b \in \Ilq}\chi(b)} \\
&= \begin{cases}
\ell^{\mlq}(\ell-1), & \text{if } \chi=\chi_0, \\
0, & \text{if } \chi\ne\chi_0,
\end{cases}
\Biggr\} - {\sum_{b \in \Hlq}\chi(b)} + 1
\end{align*}
by orthogonality. Comparing this expression to the desired formula, it suffices to prove that
\[
{\sum_{b \in \Hlq}\chi(b)} = \begin{cases}
\ell^{\mlq}, & \text{if } \chi^{\ell-1}=\chi_0, \\
0, & \text{if } \chi^{\ell-1}\ne\chi_0.
\end{cases}
\]
But this is an immediate consequence of the fact that the subgroup $\{\chi\mod{\ell^{\mlq}} \colon \chi^{\ell-1} = \chi_0\}$ is the dual group to the subgroup $\Hlq$ which is the set of $(\ell-1)$st powers modulo $\mlq$. More concretely, if $\chi^{\ell-1}=\chi_0$, then for every~$b \in \Hlq$ there exists $a\in\Mlq$ such that $a^{\ell-1}\equiv b\mod {\ell^{\mlq+1}}$, and therefore $\chi(b) = \chi(a^{\ell-1}) = \chi^{\ell-1}(a) = \chi_0(a)=1$ for each of the $\ell^{\mlq}$ elements of $\Hlq$. On the other hand, if $\chi^{\ell-1}\ne\chi_0$, then we can choose $c\in\Mlq$ such that $\chi^{\ell-1}(c)\ne1$; then, since $c^{\ell-1}$ is in the subgroup~$\Hlq$,
\[
{\sum_{b \in \Hlq}\chi(b)} = {\sum_{b \in \Hlq}\chi(c^{\ell-1}b)} = \chi^{\ell-1}(c) {\sum_{b \in \Hlq}\chi(b)}
\]
which forces ${\sum_{b \in \Hlq}\chi(b)}=0$.
\end{proof}

By Definition~\ref{notation 2.1} and Remark~\ref{Bq Blq remark}, we can write equation~\eqref{ABqs formula} as
\begin{align}
A_{\B_q}(s) &= \prod_{p \in \B_q} (1 - p^{-s})^{-\phi(Q_q)} \cdot \prod_{\chi \mod{Q_q}} \prod_p (1 - \chi(p) p^{-s})^{\prod_{\ell<q} \sum_{b_\ell\in \Blq} \chi_\ell(b_\ell)} \notag \\
&= \prod_p \biggl( \begin{cases}
(1 - p^{-s})^{-\phi(Q_q)}, & \text{if } p\in\B_q, \\
1, & \text{if } p\notin B_q
\end{cases} \Biggr\} \prod_{\chi \mod{Q_q}} (1 - \chi(p) p^{-s})^{\prod_{\ell<q} \sum_{b_\ell\in \Blq} \chi_\ell(b_\ell)} \biggr). \label{okok}
\end{align}
The factor corresponding to a prime~$p$ is some rational function of $p^{-s}$ that depends only upon the residue class of $p\mod{Q_q}$. When we compute these $\phi(Q_q)$ rational functions, the following lemma implies that each one is actually a rational function of $p^{-ks}$ for some integer $k\ge2$ (depending on the residue class). Consequently, while the above formulas are valid for $\Re s>1$, the entire expression for $A_{\B_q}(s)$ will end up converging for $\Re s>\frac12$ at least; in particular, we can evaluate $A_{\B_q}(1)$ by simply plugging $s=1$ into the resulting expression.

\begin{lemma} \label{product shape}
Given $Q\in\N$, let~$H$ be any subgroup of the group of Dirichlet characters modulo~$Q$. If $p\nmid Q$ and the cardinality of $\{\chi(p)\colon \chi\in H\}$ equals~$k$, then
\[
\prod_{\chi\in H} (1-\chi(p)p^{-s}) = (1-p^{-ks})^{(\#H)/k}.
\]
\end{lemma}

\begin{proof}
Since the image $\{\chi(p)\colon \chi\in H\}$ is a subgroup of $\C^\times$, it must equal the cyclic group of $k$th roots of unity. It follows easily that the image contains all $k$th roots of unity each with equal multiplicity $(\#H)/k$. Therefore
\[
\prod_{\chi\in H} (1-\chi(p)p^{-s}) = \prod_{j=0}^{k-1} (1-e^{2\pi ij/k}p^{-s})^{(\#H)/k} = (1-p^{-ks})^{(\#H)/k}
\]
by the factorization of $1-z^k$ over the complex numbers.
\end{proof}

\section{Calculating the leading constant for $q=4$} \label{C4 section}

We now have all the tools we need to calculate the constant~$C_4$ from Theorem~\ref{q=3,4,5 theorem}.

\begin{lemma} \label{AB4 lemma}
For $\Re s>1$ we have
\begin{multline} \label{ABs2}
A_{\B_4}(s) = \prod_{p \equiv 5,29 \mod{36}} (1-p^{-6s})^2 (1-p^{-2s})^{-6} \prod_{p \equiv 7,11,23,31 \mod{36}} (1-p^{-6s})^2 \\
\times \prod_{p \equiv 13,25 \mod{36}} (1-p^{-3s})^4 \prod_{p \equiv 19,35 \mod{36}} (1-p^{-2s})^6.
\end{multline}
\end{lemma}

\noindent Note that the products on the right-hand side converge absolutely for $\Re s>\frac12$.

\begin{proof}
When $q=4$, we have $Q_4 = \prod_{\ell<4} \ell^{\mlq+1} = 2^{1+1}3^{1+1} = 36$; we also have $\B_{2,4} = \{1\mod4\}$ and $\B_{3,4} = \{1,2,5,8\mod9\}$ and therefore $\B_4 = \{1,5,17,29\mod{36}\}$. Equation~\eqref{okok} then simplifies since the sum over $\B_{2,4}$ always equals~$1$:
\begin{align}
A_{\B_4}(s) = \prod_p \biggl( \begin{cases}
(1 - p^{-s})^{-12}, & \text{if } p\in\B_4, \\
1, & \text{if } p\notin B_4
\end{cases} \Biggr\} \prod_{\chi \mod{36}} (1 - \chi(p) p^{-s})^{\sum_{b_3\in \B_{3,4}} \chi_3(b_3)} \biggr),
\end{align}
where $\chi_3$ is the character\mod9 that is the restriction of~$\chi$ to the residue class $1\mod4$.
By Lemma~\ref{lemma 5.1}, this becomes
\begin{multline*}
A_{\B_4}(s) = \prod_p \biggl( \begin{cases}
(1 - p^{-s})^{-12}, & \text{if } p\in\B_4, \\
1, & \text{if } p\notin \B_4
\end{cases} \Biggr\} \\
\times \prod_{\substack{\chi \mod{36} \\ \chi_3=\chi_0}} (1 - \chi(p) p^{-s})^6
\prod_{\substack{\chi \mod{36} \\ \chi_3^2=\chi_0}} (1 - \chi(p) p^{-s})^{-3}
\prod_{\chi \mod{36}} (1 - \chi(p) p^{-s}) \biggr).
\end{multline*}
The products on the second line are all (powers of) products of the shape treated by Lemma~\ref{product shape}, over character subgroups of cardinality~$2$, $4$, and~$12$, respectively; the expression evaluates to
\begin{multline} \label{cases needed}
A_{\B_4}(s) = \prod_p \biggl( \begin{cases}
(1 - p^{-s})^{-12}, & \text{if } p\in\B_4, \\
1, & \text{if } p\notin B_4
\end{cases} \Biggr\} \\
\times (1 - p^{-k_1s})^{12/k_1} (1 - p^{-k_2s})^{-12/k_2} (1 - p^{-k_3s})^{12/k_3} \biggr),
\end{multline}
where $k_1 = \#\{ \chi(p)\colon \chi\mod{36},\, \chi_3=\chi_0\}$ and $k_2 = \#\{ \chi(p)\colon \chi\mod{36},\, \chi_3^2=\chi_0\}$ and $k_3 = \#\{ \chi(p)\colon \chi\mod{36}\}$, quantities that depend only upon $p\mod{36}$. A case-by-case analysis of the~$12$ residue classes in $M_{36}$ results in the expression asserted in the lemma. To see two example cases:
\begin{itemize}
\item The residue class $p\equiv29\mod{36}$, which is $1\mod4$ and $2\mod9$, is in $\B_4$. In this case
equation~\eqref{cases needed} becomes $(1-p^{-s})^{-12} (1-p^{-s})^{12} (1-p^{-2s})^{-6} (1-p^{-6s})^2$
since $k_1=1$, $k_2=2$, and $k_3=6$.
\item The residue class $p\equiv31\mod{36}$, which is $3\mod4$ and $4\mod9$, is not in $\B_4$. In this case
equation~\eqref{cases needed} becomes $1(1-p^{-2s})^{6} (1-p^{-2s})^{-6} (1-p^{-6s})^2$
since $k_1=2$, $k_2=2$, and $k_3=6$. \qedhere
\end{itemize}
\end{proof}

\begin{remark}
The quantity $A_{\B_4}(1)$ can be calculated numerically to several decimal places by plugging in $s=1$ to equation~\eqref{ABs2} and approximating the resulting absolutely convergent products numerically using all the primes up to some bound.
\end{remark}

\begin{proof}[Proof of Theorem~\ref{q=3,4,5 theorem} for $q=4$]
Equation~\eqref{GBqs formula} becomes
\[
G_{\B_4}(s)^{12} = A_{\B_4}(s) (1 - 2^{-s})^4(1 - 3^{-s})^4 \prod_{\substack{\chi \mod{36} \\ \chi \ne \chi_0}} L(s,\chi)^{{\sum_{b \in \B_4}\overline{\chi}(b)}}.
\]
Therefore by Lemma~\ref{lemma 5.1},
\[
G_{\B_4}(1) = \frac1{3^{1/3}} A_{\B_4}(1)^{1/12} \prod_{\substack{\chi \ne \chi_0 \\ \chi_3=\chi_0}} L(1,\chi)^{1/3} \prod_{\substack{\chi \ne \chi_0 \\ \chi_3^2=\chi_0}} L(1,\chi)^{-1/6} \prod_{\chi \ne \chi_0} L(1,\chi)^{1/12},
\]
where the products are over the characters\mod{36} with the given properties, remembering that~$\chi_3$ is the restriction of~$\chi$ to the residue class $1\mod4$. (As a reality check on the notation, the first product runs over all characters $\chi\mod{36}$ that are nonprincipal but for which the corresponding character $\chi_3\mod9$ is principal; there turns out to be exactly one such character, namely the character induced by the nonprincipal character $\chi_{-4}\mod4$. The second and third products run over~$3$ and~$11$ characters\mod{36}, respectively.)

Finally, by Definition~\ref{leading constant definition},
\[
C_4 = \frac{3G_{\B_4}(1)}{2\Gamma(\beta_4)} = \frac{3^{2/3}}{2\Gamma(1/3)} A_{\B_4}(1)^{1/12} \prod_{\substack{\chi \ne \chi_0 \\ \chi_3=\chi_0}} L(1,\chi)^{1/3} \prod_{\substack{\chi \ne \chi_0 \\ \chi_3^2=\chi_0}} L(1,\chi)^{-1/6} \prod_{\chi \ne \chi_0} L(1,\chi)^{1/12}
\]
which is identical to equation~\eqref{C4} by Lemma~\ref{AB4 lemma}.
\end{proof}

We briefly describe how we calculated the numerical approximation $C_4 \approx 0.4200344$. Standard computational software can calculate ${3^{2/3}}/{2\Gamma(1/3)} \approx 0.3882291057$ to any desired accuracy. Furthermore, there are classical results that allow us to write $L(1,\chi)$ as a finite sum of easily calculated quantities. For example, \cite[Theorem~9.9]{MV} tells us that if $q>1$ and $\chi$ is a primitive character modulo~$q$, then
\begin{equation} \label{primitive L 1 chi}
L(1, \chi)= \begin{cases}
\displaystyle\frac{i \pi \tau(\chi)}{q^2} \sum_{k \in M_q} \overline{\chi}(k) k, & \text{if }\chi \text { is odd,} \\
\displaystyle-\frac{\tau(\chi)}{q} \sum_{k \in M_q} \overline{\chi}(k) \log \sin \frac{k \pi}{q}, & \text{if }\chi \text { is even.}
\end{cases}
\end{equation}
where $\tau(\chi)=\sum_{a \in M_q} \chi(a) e^{2 \pi i a / q}$ is the Gauss sum. Furthermore, if $\chi$ is a Dirichlet character\mod q which is induced by the primitive character $\chi^* \mod{q^*}$, then plugging $s=1$ into \cite[equation~(9.2)]{MV} yields
\begin{equation} \label{imprimitive L 1 chi}
L(1, \chi) = L(1, \chi^*) \prod_{p \mid q} \biggl( 1- \frac{\chi^*(p)}{p} \biggr).
\end{equation}
Therefore for any nonprincipal character, $L(1,\chi)$ can be calculated with standard software to any desired accuracy. In this way we find that
\[
\prod_{\substack{\chi \ne \chi_0 \\ \chi_3=\chi_0}} L(1,\chi)^{1/3} \prod_{\substack{\chi \ne \chi_0 \\ \chi_3^2=\chi_0}} L(1,\chi)^{-1/6} \prod_{\chi \ne \chi_0} L(1,\chi)^{1/12} \approx 1.0608329542.
\]

It remains to calculate
\begin{multline*}
A_{\B_4}(1)^{1/12} = \prod_{p \equiv 5,29 \mod{36}} (1-p^{-6})^{1/6} (1-p^{-2})^{-1/2} \prod_{p \equiv 7,11,23,31 \mod{36}} (1-p^{-6})^{1/6} \\
\times \prod_{p \equiv 13,25 \mod{36}} (1-p^{-3})^{1/3} \prod_{p \equiv 19,35 \mod{36}} (1-p^{-2})^{1/2}.
\end{multline*}
We approximate $A_{\B_4}(1)^{1/12}$ by its truncation using all of the primes up to $10^7$, which evaluates to $1.0198817240$. To control the error in this approximation, we note that $\bigl| \log(1-p^{-2})^{-1/2} \bigr| \le \frac12p^{-2}$ by a linear approximation to the convex function $\log(1-t)$; the upper bounds for the absolute logarithms of the other possible Euler factors are all majorized by this one. Thus the absolute logarithm of the truncation error is at most 
\begin{align*}
\biggl| \log \prod_{p>10^7} (1-p^{-2})^{-1/2} \bigg| \le \sum_{p>10^7} \frac1{2p^2} &= \int_{10^7}^\infty \frac1{2t^2} \,d\bigl( \pi(t) - \pi(10^7) \bigr) = \int_{10^7}^\infty \frac{\pi(t) - \pi(10^7)}{t^3} \,dt
\end{align*}
after integrating by parts. Using the Rosser--Schoenfeld bound $\pi(x) < 1.25506 x/\log x$ \cite[equation~(3.6)]{RS}, we obtain the crude estimate
\[
\int_{10^7}^\infty \frac{\pi(t) - \pi(10^7)}{t^3} \,dt < \int_{10^7}^\infty \frac{1.25506t/\log t}{t^3} \,dt < \frac{1.25506}{\log10^7} \int_{10^7}^\infty \frac{dt}{t^2} = \frac{1.25506}{10^7\log10^7} < 8\cdot10^{-9}.
\]
Therefore the true value of $A_{\B_4}(1)^{1/12}$ is in the interval $(1.019881724 e^{-8\cdot10^{-9}} , 1.019881724 e^{8\cdot10^{-9}} )$.

We conclude that $C_4 \approx 0.3882291057 \cdot 1.0608329542 \cdot 1.019881724 \approx 0.4200344$ and that this approximation is accurate to the indicated seven decimal places.

\section{Calculating the leading constant for $q=5$} \label{C5 section}

As a final example, we use the same technique to calculate the constant~$C_5$ from Theorem~\ref{q=3,4,5 theorem}.

\begin{lemma} \label{AB5 lemma}
For $\Re s>1$ we have
\begin{multline} 
A_{\B_5}(s) = \prod_{\substack{p \equiv 19,35,37,53, \\ 55,71 \mod {72}}} (1-p^{-2s})^{12} \prod_{\substack{p \equiv 5,7,11,13,23,29,31,43, \\ 47,59,61,67 \mod {72}}} (1-p^{-6s})^4 \\
\times \prod_{p \equiv 41,65 \mod {72}} (1-p^{-6s})^4 (1-p^{-2s})^{12} \prod_{p \equiv 25,49 \mod {72}} (1-p^{-3s})^8 .
\end{multline}
\end{lemma}

\begin{proof}
When $q=5$, we have $Q_5 = \prod_{\ell<5} \ell^{\mlq+1} = 2^{2+1}3^{1+1} = 72$; we also have $\B_{2,5} = \{1\mod8\}$ and $\B_{3,5} = \{1,2,5,8\mod9\}$ and therefore $\B_4 = \{1,17,41,65\mod{72}\}$. Equation~\eqref{okok} then simplifies since the sum over $\B_{2,5}$ always equals~$1$:
\begin{align}
A_{\B_5}(s) = \prod_p \biggl( \begin{cases}
(1 - p^{-s})^{-24}, & \text{if } p\in\B_5, \\
1, & \text{if } p\notin B_5
\end{cases} \Biggr\} \prod_{\chi \mod{72}} (1 - \chi(p) p^{-s})^{\sum_{b_3\in \B_{3,5}} \chi_3(b_3)} \biggr),
\end{align}
where $\chi_3$ is the character\mod9 that is the restriction of~$\chi$ to the residue class $1\mod4$.
By Lemma~\ref{lemma 5.1}, this becomes
\begin{multline*}
A_{\B_5}(s) = \prod_p \biggl( \begin{cases}
(1 - p^{-s})^{-24}, & \text{if } p\in\B_5, \\
1, & \text{if } p\notin \B_5
\end{cases} \Biggr\} \\
\times \prod_{\substack{\chi \mod{72} \\ \chi_3=\chi_0}} (1 - \chi(p) p^{-s})^6
\prod_{\substack{\chi \mod{72} \\ \chi_3^2=\chi_0}} (1 - \chi(p) p^{-s})^{-3}
\prod_{\chi \mod{72}} (1 - \chi(p) p^{-s}) \biggr).
\end{multline*}
The products on the second line are all (powers of) products of the shape treated by Lemma~\ref{product shape}, over character subgroups of cardinality~$4$, $8$, and~$24$, respectively; the expression evaluates to
\begin{multline} \label{cases needed}
A_{\B_5}(s) = \prod_p \biggl( \begin{cases}
(1 - p^{-s})^{-24}, & \text{if } p\in\B_5, \\
1, & \text{if } p\notin B_5
\end{cases} \Biggr\} \\
\times (1 - p^{-k_1s})^{24/k_1} (1 - p^{-k_2s})^{-24/k_2} (1 - p^{-k_3s})^{24/k_3} \biggr),
\end{multline}
where $k_1 = \#\{ \chi(p)\colon \chi\mod{72},\, \chi_3=\chi_0\}$ and $k_2 = \#\{ \chi(p)\colon \chi\mod{72},\, \chi_3^2=\chi_0\}$ and $k_3 = \#\{ \chi(p)\colon \chi\mod{72}\}$, quantities that depend only upon $p\mod{72}$. A case-by-case analysis of the~$24$ residue classes in $M_{72}$ results in the expression asserted in the lemma.
\end{proof}

\begin{proof}[Proof of Theorem~\ref{q=3,4,5 theorem} for $q=5$]
Equation~\eqref{GBqs formula} becomes
\[
G_{\B_5}(s)^{24} = A_{\B_5}(s) (1 - 2^{-s})^4(1 - 3^{-s})^4 \prod_{\substack{\chi \mod{72} \\ \chi \ne \chi_0}} L(s,\chi)^{{\sum_{b \in \B_5}\overline{\chi}(b)}}.
\]
Therefore by Lemma~\ref{lemma 5.1},
\[
G_{\B_5}(1) = \frac1{3^{1/6}} A_{\B_5}(1)^{1/24} \prod_{\substack{\chi \ne \chi_0 \\ \chi_3=\chi_0}} L(1,\chi)^{1/6} \prod_{\substack{\chi \ne \chi_0 \\ \chi_3^2=\chi_0}} L(1,\chi)^{-1/12} \prod_{\chi \ne \chi_0} L(1,\chi)^{1/24},
\]
where the products are over the characters\mod{72} with the given properties, remembering that~$\chi_3$ is the restriction of~$\chi$ to the residue class $1\mod8$.

Finally, by Definition~\ref{leading constant definition},
\[
C_5 = \frac{3G_{\B_5}(1)}{2\Gamma(\beta_5)} = \frac{3^{5/6}}{2\Gamma(1/6)} A_{\B_5}(1)^{1/24} \prod_{\substack{\chi \ne \chi_0 \\ \chi_3=\chi_0}} L(1,\chi)^{1/6} \prod_{\substack{\chi \ne \chi_0 \\ \chi_3^2=\chi_0}} L(1,\chi)^{-1/12} \prod_{\chi \ne \chi_0} L(1,\chi)^{1/24}
\]
which is identical to equation~\eqref{C5} by Lemma~\ref{AB5 lemma}.
\end{proof}

As before, standard computational software can calculate ${3^{5/6}}/{2\Gamma(1/6)} \approx 0.2095133578$, and the known formulas for $L(1,\chi)$ allow us to calculate
\[
\prod_{\substack{\chi \ne \chi_0 \\ \chi_3=\chi_0}} L(1,\chi)^{1/6} \prod_{\substack{\chi \ne \chi_0 \\ \chi_3^2=\chi_0}} L(1,\chi)^{-1/12} \prod_{\chi \ne \chi_0} L(1,\chi)^{1/24} \approx 0.9354960209.
\]
Finally, the method of the previous section allows us to approximate
\begin{align*}
A_{\B_5}(1)^{1/24} &= \prod_{\substack{p \equiv 19,35,37,53, \\ 55,71 \mod {72}}} (1-p^{-2})^{1/2} \prod_{\substack{p \equiv 5,7,11,13,23,29,31,43, \\ 47,59,61,67 \mod {72}}} (1-p^{-6})^{1/6} \prod_{p \equiv 25,49 \mod {72}} (1-p^{-3})^{1/3} \\
&\qquad{}\times \prod_{p \equiv 41,65 \mod {72}} (1-p^{-6})^{1/6} (1-p^{-2})^{1/2} \approx 0.9980828307,
\end{align*}
and the error analysis from the previous section is valid without change for this product. We conclude that $C_5 \approx 0.2095133578 \cdot 0.9354960209 \cdot 0.9980828307 \approx 0.2095134$ accurate rounded to the indicated seven decimal places.

\section*{Acknowledgments}

The first author was supported in part by a Natural Sciences and Engineering Council of Canada Discovery Grant.

\end{document}